 \documentclass{amsart}
\usepackage[utf8]{inputenc}
\usepackage{graphicx}
\usepackage{amsrefs}
\usepackage{amssymb}
\usepackage{mathtools}
\usepackage[matrix,arrow,ps]{xy}
\usepackage{hyperref}
\usepackage[capitalize]{cleveref}
\usepackage{tikz}
\usepackage{tikz-cd}

\usepackage{faktor}
\usepackage{verbatim}
\usepackage{epstopdf}

\newcommand{\hol}{\mathfrak{hol}}

\newcommand{\C}{\mathbb{C}}
\newcommand{\CN}{\mathbb{C}^N}
\newcommand{\CNp}{\mathbb{C}^{N^\prime}}

\newcommand{\R}{\mathbb{R}}

\newcommand{\N}{\mathbb{N}}

\newcommand{\nequiv}{{\equiv \!\!\!\!\!\!  / \,\,}}

\newcommand{\holmaps}{\mathcal{H}}

\newcommand{\cps}[1]{\C\{#1\}}
\newcommand{\cpstwo}[2]{#1 \{#2\}}

\newcommand{\todo}[1]{}

\newcommand{\parcur}[1]{\mathfrak P^{#1}}
\newcommand{\tz}{\mathbf t_0}

\newlength{\extendaxesby}

\DeclareMathOperator{\rank}{rk}

\DeclareMathOperator{\real}{Re}
\DeclareMathOperator{\im}{Im}
\DeclareMathOperator{\re}{Re}

\newtheorem*{theorem*}{Theorem}
 \newtheorem{thm}{Theorem}
\newtheorem{theorem}[thm]{Theorem}

\newtheorem{lem}[thm]{Lemma}
\newtheorem{lemma}[thm]{Lemma}

\newtheorem{cor}[thm]{Corollary}
\newtheorem{corollary}[thm]{Corollary}

\theoremstyle{definition}

\newtheorem{definition}[thm]{Definition}
\newtheorem{exa}{Example}
\newtheorem{example}[exa]{Example}

\newtheorem{remark}[thm]{Remark}

\begin{document}

\author{Giuseppe della Sala}
\address{Department of Mathematics, American University of Beirut}
\email{gd16@aub.edu.lb}
\author{Bernhard Lamel}
\address{Fakult\"at f\"ur Mathematik, Universit\"at Wien}
\email{bernhard.lamel@univie.ac.at}
\author{Michael Reiter}
\address{Fakult\"at f\"ur Mathematik, Universit\"at Wien}
\email{m.reiter@univie.ac.at}

\begin{abstract}
We study the deformation theory of CR maps in the positive 
codimensional case. In particular, we study structural properties
of the {\em mapping locus} $E$ of (germs of nondegenerate) holomorphic 
maps $H \colon (M,p) \to M'$ 
between generic real submanifolds $M \subset \CN$ and 
$M' \subset \CNp$, defined to be the set of points 
$p' \in M'$ which admit such a map with $H(p) = p'$.  
We  show  that this set $E$ is semi-analytic
and provide examples for which $E$ posseses (prescribed) singularities.
\end{abstract}


\title{Deformations of CR maps and applications}

\maketitle


\section{Introduction} 
\label{sec:introduction}
For a holomorphic map 
$H \colon (\CN,p) \to \CNp$, it follows from an easy extension 
of Poincar\'e's counting argument \cite{Poincare07} that given 
a germ of a real-analytic submanifold $(M,p)$, ``most'' real 
analytic subvarieties
$M'\subset \CNp$ cannot be deformed 
biholomorphically to a variety containing $H(M)$. Therefore, 
given a germ of a real-analytic submanifold $(M,p)\subset (\CN,p)$
and a subvariety $M'\subset \CNp$, one can expect that only 
few points $p'\in M'$ allow for a (nonconstant) holomorphic map 
$H\colon (\CN,p) \to (\CNp,p') $ satisfying $H(M)\subset M'$. 
This leads us to define the {\em mapping locus}
\[ E = \left\{ p'\in M' \colon \exists H\colon (\CN,p) \to (\CNp,p')\text{ holomorphic }, H(M)\subset M', H \nequiv p'   \right\}. \]
In this paper we study properties of this set, which was 
introduced in the equidimensional setting 
where one considers {\em biholomorphisms} in the paper
 \cite{dSJL15} as the  {\em equivalence locus} $E(p)$ of $p\in M$. 
Our results are valid for maps $H$ which are 
usually called ``nondegenerate'' maps.

So let us give the necessary definitions for stating our results.
We denote 
the ideal associated to $M'$ at the point $p'$, consisting of all 
germs of real analytic functions vanishing on $M'$, by 
$\mathcal{I}_{p'} (M') \subset \cps{w-p', \overline{w - p'}}$, 
and the set of 
all germs of real-analytic CR vector fields tangent to $M$ near $p$ by 
$\Gamma_p (M)$. We say that such a map $H$ is  {\em $\ell$-finitely nondegenerate} at $p$ if 
\[ \dim_\C \left\{ \bar L_1 \cdots \bar L_k \varrho'_w (H(z), \overline{H(z)})|_{z=p}  \colon \bar L_j \in \Gamma_p(M), k\leq \ell , \,\varrho' \in \mathcal{I}_{H(p)} (M') \right\} = N'.\]

Given $M$ and $M'$ as above, and $\ell \in \N$, we define
the ($\ell$-finitely nondegenerate) {\em mapping locus} $E_\ell \subset E \subset M'$
consisting of all points $p' \in M'$ with the property that there exists an 
 $\ell$-finitely nondegenerate map $H\colon (M,p) \to M'$ with $H(p) = p'$. 

The rigidity properties 
 of real objects with respect to holomorphic maps already alluded to 
 above lead to interesting structural 
 properties this set. Apart from this, 
 there are several reasons motivating the study of $E$.
One
of the main possible applications is in the 
study of the moduli space of CR maps 
with respect to the actions of the automorphism groups of 
 $M$ and $M'$ from the right or the left, respectively. This is particularly
 interesting (and has been studied a lot) in 
 the case where $M$ and $M'$ are spheres. We refer the interested reader 
 to the survey of Huang and Ji \cite{HJ07} on this matter, and also 
 note that there is a notion of homotopy equivalence in this setting
 introduced by D'Angelo and Lebl \cite{DL16}. 

Another motivation is that when one studies deformations 
of {\em proper maps} between domains with 
real-analytic boundaries, then the existence of such
maps finds obstructions in the existence of 
maps between the boundaries very 
naturally, as all such maps extend holomorphically across the boundary 
in many settings (see e.g. the paper by Mir \cite{MR3667415}). Let 
us now state our first theorem:

\begin{theorem}
\label{thm:newmain} The mapping locus $E_\ell\subset M'$ is locally a semi-analytic set. 
\end{theorem}

In particular, we recover one 
of the main results of \cite{dSJL15} for the 
equivalence locus $E(p)$: since it 
is also homogeneous (by definition),  $E(p)\subset M$ is necessarily a locally closed submanifold.

One might wonder whether the mapping locus $E_{\ell}$, in contrast 
to the equidimensional case,  can 
have singularities if the codimension $N' - N$ is positive.  
We construct an example showing that this is 
actually the case (even if the source manifold is assumed to be 
very nice).
\begin{theorem}
\label{thm:singularLocus}
Let $M$ be the unit sphere in $\C^N$. Then there is a real hypersurface $M' \subset \C^{N+1}$ such that the mapping locus $E_{\ell}$, for any $\ell\geq 2$, is a singular subset of $M'$.
\end{theorem}

Our approach to studying the mapping locus is 
to consider the variation of the image point $p' \in M'$ as 
a deformation of $M'$ and to deduce the semi-analyticity result 
from a more general semi-analyticity result valid for general 
deformations of $M'$.

This approach allows us 
to shed additional light on the mapping locus in 
some interesting cases. We point out one instance of this 
here: the degeneration 
of the mapping locus to a point can be checked by a sufficient 
linear criterion, which allows us to recover a statement already 
implicit in results of \cite{dSLR17} from our considerations of 
deformations.  

We recall the necessary definitions: 
Assume that $H\colon (M,p) \to (M',p')$ 
is an $\ell$-finitely nondegenerate map. We say that 
$Y(z) \in \cps{z-p}^{N'}$ is an infinitesimal deformation of $H$ if 
$Y(p) = 0$ and if
\[ \real \left(   \varrho_w (H(z), \overline{H(z)}) \cdot Y(z)  \right)\big|_{z\in M} = 0.   \] 
We note that we shall see later 
that the set of infinitesimal deformations of $H$ is 
a (finite-dimensional) real subspace $\hol(H)$ of $\cps{z-p}^{N'}$, 
which is tightly related to the 
tangent space of the set of possible maps of $(M,p)$ into $M'$. 

\begin{theorem}
\label{thm:infdefmaplocus} Assume $H$ is $\ell$-finitely nondegenerate and 
that there exist no nontrivial infinitesimal deformations of $H$, i.e. $\dim_\R \hol(H) = 0$. Then there exists a neighbourhood
$U$ of $p'$ such that  $E_\ell \cap U = \{ p' \}$. 
\end{theorem}

\cref{thm:newmain} and \cref{thm:infdefmaplocus} are obtained by a more general study of analytic deformations of the target manifold $M'$ together with the jet parametrization technique for CR maps, see \cref{sec:prelim} for all relevant definitions. Their proofs are given in \cref{sec:propLocus} and \cref{section:mapsJPP}. The proof of \cref{thm:singularLocus} is given in \cref{sec:singularlocus}. Finally, \cref{sec:examples} contains a list of examples illustrating various properties of the mapping locus.

\section{Preliminaries and further results}
\label{sec:prelim}

This section introduces all the relevant notions and notations used throughout the paper; we also state the more general theorems from which 
we will deduce the theorems stated in the introduction.

\subsection{Manifolds, maps and deformations}
Let $\mathcal H((\C^N,p),\C^{N'})$ be the set of germs at $p$ of maps from $\C^N$ to $\C^{N'}$. This space is endowed with the inductive limit topology with respect to the Banach spaces $\mathcal H((\overline{B_R(p)},p),\C^{N'})$, where $B_R(p)$ denotes the ball of radius $R>0$ in $\C^N$ centered at $p$. In the following every subspace of $\mathcal H((\C^N,p),\C^{N'})$ will be equipped with the induced topology.


We define $\mathcal H((\C^N,p),(\C^{N'},p'))\subset \mathcal H((\C^N,p),\C^{N'})$ as the subset of $\mathcal H((\C^N,p),\C^{N'})$ of maps $H$ satisfying $H(p)=p'$.

Let $M \subset \C^N$ be a generic real-analytic submanifold and  $M' \subset \C^{N'}$ a real-analytic subvariety. A holomorphic map $H: M \rightarrow M'$ can be considered as the  restriction of a holomorphic map $H$ defined on a neighborhood of $M$. We denote by $\mathcal H(M,M')$ the collection of all holomorphic maps sending $M$ into $M'$. 

If $(M,p) \subset \C^N$ is a germ of a real-analytic submanifold and $(M',p') \subset \C^{N'}$ is a germ of a real-analytic subvariety  we denote by $\mathcal H((M,p),(M',p'))$ the collection of holomorphic maps $H$ sending $(M,p)$ into $(M',p')$ (in particular $H(p) = p'$).  

Often we need to consider subsets of $\mathcal H(M,M')$ or $\mathcal H((M,p),(M',p'))$ satisfying certain in some sense good geometric and analytic properties, especially those admitting a jet parametrization. An example is given by the class of \textit{finitely nondegenerate maps} (see \cref{section:mapsJPP}). We will generically denote such a subset of maps by $\mathcal F \subset \mathcal H(M,M')$ or $\mathcal F \subset \mathcal H((M,p),(M',p'))$.

In the paper we will need to treat not only the case of a fixed target set 
$M'$, but also the case of a \textit{deformation} of $M'$; that is a family of subvarieties $M'_\epsilon$ including $M'$. 
More precisely, we extend a definition taken from \cite{dSJL15}; we assume that the reader is 
acquainted with the basics of semi-analytic geometry, but refer to \cref{sec:realGeometry} for 
a discussion of the notions we use here.

\begin{definition}
\label{def:deformation}
Let $X$ be a semi-analytic compact set in some $\R^m$. Let $(M',p')$ be a germ of a real-analytic subvariety of $\C^{N'}$. 
A \textit{deformation}
 $\mathfrak D  = (M'_\epsilon,p')_{\epsilon \in X}$ of $(M',p')$ is given by 
a family of subvarieties $M'_\epsilon \subset \C^{N'}$, depending 
analytically on $\epsilon \in X$ in the sense that 
there exist $\varrho_1 (w, \bar w,\epsilon) , \dots , \varrho_d ( w , \bar w,\epsilon) \in \cpstwo{\mathcal{C}^\omega (X)} {w - p',\overline{w-p'}}$ such that 
\[ \mathcal{I}_{p'} (M'_\varepsilon) = \left( \varrho_1 (w, \bar w, \epsilon), \dots ,  \varrho_d (w, \bar w, \epsilon) \right),  \]
and  $M'_{\epsilon_0} = M'$ for some $\epsilon_0 \in X$. We write 
$\mathcal{I} (\mathfrak{D})$ for the set of all 
$\varrho (w, \bar w, \epsilon) \in \cpstwo{\mathcal{C}^\omega (X)} {w - p',\overline{w-p'}}$ satisfying $\varrho( \cdot, \cdot, \epsilon) \in \mathcal{I}_{p'} (M'_\epsilon) $ for every $\epsilon \in X$.


 A \textit{base-point-type deformation} of $M'$ is a deformation obtained in the following way: We choose $r>0$ small enough, 
take $X = M'\cap \overline{B_r (p')}$, 
 and for all $q' \in X$ we define 
 the germ $(M'_{q'},p')$ as $(M'+p'-q',p')$, where $M' + w' = \{v'+w': v' \in M'\}$, which is  a deformation in our sense, 
 as we can use 
 $\varrho_j (w, \bar w, q') = \varrho_j (w-(p'- q'), \overline{w -( p'-q')}) $ for any generating set $\varrho_1 , \dots ,\varrho_d \in \mathcal{I}_{p'} (M')$.

\end{definition}

It is useful to study the infinitesimal notion corresponding to deformations of maps; compare with \cite{dSLR15a,dSLR15b,dSLR17,CH02}.

Let $H: M \rightarrow \CNp$ be a real-analytic CR map satisfying
$H(M)\subset M'$. We denote by $\Gamma_H = \Gamma_{CR}(H^{*}(\mathbb C T(\mathbb C^{N'})))$ the space of real-analytic CR sections of the pull back bundle of $\mathbb C T(\mathbb C^{N'})$ with respect to $H$, cf. \cite{dSLR17}.

\begin{definition}
\label{def:infdefdef}
Let $(M,p)$ and $\mathfrak D =((M'_{\epsilon})_{\epsilon \in X},p')$ be as above with $M'_{\epsilon} = \{\rho'(\cdot, \cdot,\epsilon) = 0\}$. Let $\epsilon_0 \in X^{{\rm reg}}$ and $H_{\epsilon_0}: (M,p) \rightarrow (M'_{\epsilon_0},p')$ be in $\mathcal F$. We say that an element $(v,Y) \in T_{\epsilon_0} X \times \Gamma_{H_{\epsilon_0}}$ is an \textit{infinitesimal deformation} of $H_{\epsilon_0}$ into $\mathfrak{D}$ if $Y(p)=0$ and the following equations are satisfied:
\begin{align*}
2\re\left(\varrho_{w}\left(H_{\epsilon_0}(Z),\bar H_{\epsilon_0}(\bar Z),\epsilon_0\right)\cdot Y(Z) \right) + \varrho_{\epsilon}\left(H_{\epsilon_0}(Z),\bar H_{\epsilon_0}(\bar Z),\epsilon_0\right) \cdot v = 0, 
\end{align*}
for all $\varrho \in \mathcal{I} (\mathfrak{D})$ and $Z\in M$.
We denote the space of all infinitesimal deformations of $H_{\epsilon_0}$ into $\mathfrak{D}$ by $\hol(H_{\epsilon_0},\mathfrak{D})$. 
\end{definition}

\begin{remark}
\label{rem:derivDef}
If we consider a curve $(\epsilon(t),H(t))$ with $H(t): (M,p) \rightarrow (M'_{\epsilon(t)},p')$ for $\epsilon(t) \in X$ then $(v,Y) = (\epsilon'(0),\frac{d}{dt}|_{t=0} H(t))$ (note that $(\frac{d}{dt}|_{t=0} H(t))(p) = 0$) belongs to $\hol(H(0),\mathfrak{D})$.
The proof is the same as \cite[Lemma 21]{dSLR15a}.
\end{remark}

\begin{definition}
\label{def:parcur}
Let $H(t)\subset \holmaps(M, \C^{N'})$ be a smooth curve such that $H(0)\in \holmaps(M,M'_{\epsilon_0})$ and $\epsilon(t)$ a smooth curve in $X$ with $\epsilon(0)=\epsilon_0$. We say that $(\epsilon(t),H(t))$ is \emph{tangent to $\holmaps(M,M'_\epsilon)$ to order $r$} at $(\epsilon_0,H(0))$ if for any local 
parametrization $Z(s)$ of $M$  we have that $\varrho(H(Z(s),t), \overline{H(Z(s),t)},\epsilon(t))=O(t^{r+1})$ for any $\varrho \in \mathcal I_{H(0)}(M'_\epsilon)$. We denote the set of such parametrized curves by $\parcur{r}$ (or $\parcur{r}_{(\epsilon_0,H)}$ if we need to emphasize that $\epsilon(0)=\epsilon_0$ and $H=H(0)$).
\end{definition}

\begin{definition}\label{def:higherorderinfdef} 
Let $(M,p)$ and $\mathfrak D =((M'_{\epsilon})_{\epsilon \in X},p')$ be as above. Let $(\epsilon,H) \in X^{\rm reg} \times \mathcal F$.
We say that $(w,Y) \in (T_{\epsilon_0} X)^k \times \Gamma_H^k$, where $\Gamma_H^k = \Gamma_H \times \dots \times \Gamma_H$, is an  \textit{infinitesimal deformation of $H$ of order $k$}, and write  $(w,Y)\in \hol^k(H,\mathfrak D)$ if $\tau_k(w,Y) \coloneqq (\epsilon + t w_1 + \cdots t^k w_k,H + t Y_1 + \cdots + t^k Y_k) \in \mathcal H_{(\epsilon,H)}[t]\cap \parcur{k}_{(\epsilon,H)}$.
\end{definition}

Note that for $k=1$ we recover $\hol(H,\mathfrak{D})$ given in \cref{def:infdefdef}.

In the next lemma we use the following identity: If $\sigma$ is a real-valued function on $\C^n \cong \R^{2n}$ we can write $\sigma(p) = \tilde\sigma(p,\bar p)$ for $p\in \C^n$. Then for any $v\in \C^n$ we have $\sigma_p \cdot v = \tilde\sigma_p \cdot v + \tilde \sigma_{\bar p} \cdot \bar v$, where the first $\cdot$ denotes the inner product in $\R^{2n}$ (hence $v$ has been written as a vector in $\R^{2n}$) and the second $\cdot$ is given by $a \cdot b \coloneqq a_1 b_1 + \cdots + a_n b_n$ for $a=(a_1,\ldots, a_n) \in \C^n$ and $b=(b_1,\ldots, b_n) \in \C^n$. 

\begin{lemma}
\label{lem:differentHol}
Let $(M,p)$ be as above, and consider the base-point-type deformation of $M'$ and $H:M \rightarrow M'$ a holomorphic map. Let $\hol(H)$ be as in \cite{dSLR17}. Then $(v,Y) \in \hol(H,\mathfrak D)$ if and only if $Y+v \in \hol(H)$. Moreover, the map $\hol(H,\mathfrak D) \ni (v,Y) \mapsto Y+v \in \hol(H)$ is an isomorphism.
\end{lemma}

\begin{proof}
Assume $p=0$ and $p'=0$ and denote by $\varrho_1,\ldots, \varrho_d$ a set of generators of $\mathcal I_0(M')$. By definition $\rho_j(Z',\bar Z',\epsilon) = \varrho_j(Z'+\epsilon,\bar Z'+\bar \epsilon)$ for $\epsilon\in M'$ and $1\leq j \leq d$ is a set of generators for $\mathcal I_0(M'_\epsilon)$.
Since ${\rho_j}_{Z'} = {\varrho_j}_{Z'}$ and ${\rho_j}_{\epsilon} =  2 \re ({\varrho_j}_{Z'})$ it holds that $(v,Y) \in \hol(H,\mathfrak{D})$ if and only if for $1\leq j\leq d$
\begin{align*}
& 2\re\left({\rho_j}_{Z'}\left(H(Z),\bar H(\bar Z),\epsilon\right)\cdot Y(Z) \right) + {\rho_j}_{\epsilon}\left(H(Z),\bar H(\bar Z),\epsilon\right) \cdot v = 0, \quad Z\in M,  \\
\Longleftrightarrow \quad  & 2\re\left({\varrho_j}_{Z'}\left(H(Z),\bar H(\bar Z)\right)\cdot (Y(Z)+v) \right) = 0, \quad Z\in M, 1\leq j \leq d.
\end{align*}
The last equation is satisfied if and only if $Y+v \in \hol(H)$. The last statement follows from the fact that $\hol(H) \ni W \mapsto (W(0),W+W(0)) \in \hol(H,\mathfrak{D})$ is an inverse to the map given in the hypothesis.
\end{proof}

\begin{corollary}
\label{cor:differentHol}
If $\pi_1$ is the projection on the first factor, then $\pi_1(\hol(H,\mathfrak{D})) = \hol(H)(p) = \{X(p): X \in \hol(H)\}$.
\end{corollary}

We will be particularly interested in studying the set of deformation parameters $\epsilon$ for which a map between $M$ and $M'_\epsilon$ exists. We will call this set the \textit{mapping locus}, more precisely we have the following definition, cf. the definition of equivalence locus from \cite{dSJL15} for the equidimensional case.

\begin{definition}\label{def:maplocus}
Let $(M,p)$ be a germ of submanifold of $\C^N$, and let $(M'_\epsilon,p')_{\epsilon\in X}$ be a  deformation of $(M',p')=(M'_{\epsilon_0},p')\in \C^{N'}$. Let $\mathcal F_\epsilon \subset \mathcal H((M,p),(M'_\epsilon,p'))$ and $\mathcal F = (\mathcal F_\epsilon)_{\epsilon \in X}$.  We define the \emph{$\mathcal F$-mapping locus} as the set $E_{\mathcal F}\subset X$ given by 
\[E_{\mathcal F}=\{\epsilon\in X| \mathcal F_{\epsilon} \neq \emptyset\}.\]
In other words, $\epsilon \in E_{\mathcal F}$ if and only if there exits a holomorphic map $H: M \rightarrow M'_{\epsilon}$ with $H(p) = p'$ and $H \in \mathcal F_{\epsilon}$.
In particular if we consider base-point-type deformations, then $E_{\mathcal F}\subset M'$. 
\end{definition}

\subsection{Jet spaces}
We will work with maps by means of their jet through suitable parametrization results. The following definitions are very standard and here we mainly aim at establishing the notation used later in the paper. 
For all $p =(p_1,\ldots, p_N)\in \C^N$ we define the space of $k$-jets at $p$ of holomorphic maps $\mathbb C^N\to \mathbb C^{N'}$ as follows:
\begin{align*}
J_p^k = \faktor{\mathbb C \{Z-p\}^{N'} }{\mathfrak{m}_p^{k+1}},
\end{align*}
where $\mathfrak{m}_p=(Z_1-p_1,\ldots,Z_N-p_N)$ is the maximal ideal of the ring of power series centered at $p$, and $j_p^k$ denotes the natural projection. For a given $k$, we will denote by $\Lambda$ the coordinates in $J_p^k$.

\subsection{Jet parametrization} 
\label{sub:jet_parametrization}
It turns out that our structural results hold in higher generality
than the setting discussed in the introduction. The methods apply 
equally well to understanding the structure of $\mathcal{F}$ if 
we assume that $\mathcal{F}$ satisfies the following
{\em jet parametrization property} JPP (see \cref{def:jetparamdef}). 

Jet parametrization results can be proved in a variety of different contexts and  have been used widely in the study of the structure of CR mappings, see e.g. 
\cite{BER97,BER99,LM07,JL13a,LMZ08}. 
In the following definition we abstractly define 
what we need from such a parametrization to obtain the desired 
structural results. In \cref{section:mapsJPP} we are going to consider classes of maps which satisfy the jet parametrization property.

\begin{definition}
\label{def:jetparamdef}
Let $(M,p)$ be a germ of submanifold of $\C^N$, and let $\mathfrak{D} = (M'_\epsilon,p')_{\epsilon\in X}$ be a germ of real-analytic deformation of $(M',p')=(M'_{\epsilon_0},p')\in \C^{N'}$, where $\epsilon_0$ is a distinguished parameter in $X$. For all $\epsilon\in X$ let $\mathcal F_{\epsilon}\subset \mathcal H(M,M'_\epsilon)$ be an open subset of maps. We say that $\mathcal F = (\mathcal F_{\epsilon})_{\epsilon\in X}$ satisfies the \emph{jet parametrization property of order $\tz\in \mathbb N$} if the following holds.

\

\noindent{\bfseries JPP:}\emph{  There exists an open neighborhood $V$ of $p$ in $\C^N$, an open neighborhood $W$ of $\epsilon_0$ in $X$, a finite index set $J$, real-analytic functions $q_j:W\times J_{p}^{\tz}\to \mathbb R, j \in J$ such that $q_j(\epsilon,\Lambda)$ is polynomial in $\Lambda$, and a holomorphic map $\Phi_j:\mathcal U_j\to \C^{N'}$ (where $\mathcal U_j=V\times U_j$ and $U_j=\{q_j(\epsilon,\Lambda)\neq 0\}\subset W\times J_{p}^{\tz}$) of the form
\begin{align} \label{rationalJetParam}
\Phi_j(Z,\epsilon,\Lambda) =  \sum_{\alpha\in \N_0^N} \frac{p_j^{\alpha} (\epsilon,\Lambda)}{q_j(\epsilon, \Lambda)^{d^j_\alpha}} Z^\alpha, \quad  p_j^\alpha \in \C\{\epsilon\}[\Lambda], \quad d^j_\alpha\in\N_0, \quad j\in J,
\end{align}
such that for every map $t\mapsto(\epsilon(t),H(t))$ belonging to $\parcur{r}_{(\epsilon(0),H(0))}$ there exists $j\in J$ such that the following holds for all $t$ close enough to $0$:
\begin{itemize}
\item[(a)] $(\epsilon(t),j_{p}^{\tz} H(t)) \in U_j$,
\item[(b)] $H(Z,t)|_{V} = \Phi_j(Z,\epsilon(t), j_{p}^{\tz} H(t)) + O(t^{r+1}).$
\end{itemize}}
\emph{
In particular,  there exist real-analytic functions $c_i^j:W\times J_{p}^{\tz}\to \mathbb R$, $i\in\N$, polynomial in $\Lambda$ such that
\begin{align}
\label{defEquationJetParam}
 A_\epsilon \coloneqq j_{p}^{\tz} (\mathcal F_\epsilon) =  \bigcup_{j\in J}\{ \Lambda\in J_{p}^{\tz} \colon q_j(\epsilon,\Lambda) \neq 0, \, c^j_i (\epsilon, \Lambda, \bar \Lambda) =0 \}.
 \end{align}
Define $A\subset X\times J_p^{\tz}$ as $A \coloneqq \bigcup_\epsilon (\{\epsilon\}\times A_\epsilon)$ and set $A_j = A \cap \mathcal{U}_j$. Then $\Lambda$ is the $k_0$-jet of a map $\mathcal F_{\epsilon} \ni H\colon (M,p) \rightarrow (M'_{\epsilon},p')$ if and only if $(\epsilon,\Lambda) \in A$.
\\
Furthermore  for any $(\epsilon(t),H(t))\in \parcur{r}_{(\epsilon,H)}$ with $\Lambda(t) = j_{p_k}^{\tz} H(t)$ we have for small enough $t$:
\begin{align}\label{defEquationJetParamcurves}
c^{j}_i (\epsilon(t),\Lambda(t), \bar{\Lambda}(t)) = O(t^{r+1}), \qquad i,j \in \N.
\end{align}
}
\end{definition}


\begin{remark}
\label{rem:trivial0}
Since $\mathcal F \subset \bigcup_{\epsilon \in X} \mathcal H(M,M'_\epsilon)\subset X \times \mathcal{H} (M,\CNp)$ we can equip $\mathcal F$ with the induced topology. Similarly as in \cite{dSLR15a} one can show that $\Phi_j: A_j \rightarrow \mathcal F$ is locally a homeomorphism. For more details we refer to \cite[Lemma 19]{dSLR15a}. 
\end{remark}

\begin{remark}
\label{rem:trivial1}
Let $(\epsilon,\Lambda) \in A^{\rm reg}_j$ and $w\in T_{(\epsilon,\Lambda)} A^{\rm reg}_j$ and consider a curve $c(t)=(\epsilon(t),\Lambda(t))$ in $A^{\rm reg}_j$ with $c(0) = (\epsilon,\Lambda)$ and $c'(0) = w$. Then a similar computation as in \cref{rem:derivDef} applied to $(\epsilon(t),H(t))$ with $H(t) = \Phi_j(.,\epsilon(t),\Lambda(t))$ shows that $D\Phi_j(T_{(\epsilon,\Lambda)} A^{\rm reg}_j) \subseteq \hol(H,\mathfrak{D})$ for any $(\epsilon,\Lambda)\in A^{\rm reg}_j$.
\end{remark}

\begin{remark}
\label{rem:trivial2}
In a similar way as in \cite{dSLR17}*{Remark 18} one can deduce a jet parametrization for $\hol^k(H,\mathfrak{D})$, which for $k=1$ implies that $\hol(H,\mathfrak{D})$ is finite dimensional. As in \cite[Cor. 32]{dSLR15a} one can deduce that $X\times \mathcal F_\epsilon \ni (\epsilon,H) \mapsto \dim(\hol(H,\mathfrak{D}))$ is upper semicontinuous. Applying this fact to a base-point-type deformation and using the last statement of \cref{lem:differentHol} shows that $p \mapsto \dim \hol(H)(p)$ is upper semicontinuous.
\end{remark}


\subsection{Further results}
\label{sec:further}

The results in the introduction actually hold in a more general setting. In particular \cref{thm:newmain} can be formulated for mappings which satisfy JPP.

\begin{theorem}\label{thm:semi-analytic0}
Let $(M,p) \subset \CN$ and $M'\subset \CNp$ be generic real-analytic submanifolds and assume that $\mathcal F$ satisfies JPP. Then $E_{\mathcal F}$ is locally a semi-analytic subset of $M'$.
\end{theorem}

\cref{thm:newmain} now, follows from \cref{thm:semi-analytic0} and \cref{jetparamdef}, which shows that the class of $\ell$-finitely nondegenerate maps satisfies JPP. Note that \cref{thm:semi-analytic0} can be regarded as a result for base-point type deformations and hence can be considered as a special case of the following more general result:

\begin{theorem}\label{thm:semi-analytic}
Let $(M,p)$, $(M'_\epsilon,p')$ and $\mathcal F$ be as in \cref{def:jetparamdef}. Then $E_{\mathcal F}$ is locally a semi-analytic subset of $X$.
\end{theorem}

\cref{thm:semi-analytic} is a partial generalization of \cite{dSJL15}*{Theorem 2}, its proof is given in \cref{sec:propLocus} below.

The parametrization method can also be used to provide a sufficient linear criterion to show that a given map is isolated in $\mathcal F$. The following theorem is a generalization of \cref{thm:infdefmaplocus} from the introduction:

\begin{theorem}
\label{thm:trivial}
Let $M, \mathfrak{D}$ and $\mathcal F$ be as in JPP. Fix $H: M \rightarrow M'_{\epsilon_0}$ with $H\in \mathcal F_{\epsilon_0}$. Suppose that $\dim \hol(H,\mathfrak{D}) = 0$, then $H$ is isolated in $\mathcal F$. 
\end{theorem}

The following corollary is an immediate consequence of \cref{thm:trivial} and \cref{cor:differentHol}.

\begin{corollary}
If $\dim \hol(H) = 0$, then $H$ is isolated in $\mathcal H(M,M')$.
\end{corollary}

The result can be proved as a consequence of \cref{rem:trivial1,rem:trivial2} in an analogous way as in \cite{dSLR15a}. In the following we are only outlining the main steps and refer to \cite[Lemma 23]{dSLR15a}. We need the following Lemma:

\begin{lemma}
\label{lem:dimEll}
Let $(\epsilon_0,\Lambda_0)\in A_j$, and suppose that $\dim \hol(\Phi_j(\epsilon_0,\Lambda_0),\mathfrak D) = \ell$. Then there exists a neighborhood $U$ of $(\epsilon_0,\Lambda_0) \in X \times  J_p^{\tz}$ such that, if $N\subset A_j$ is a submanifold with $N\cap U\neq \emptyset$, then $\dim N \leq \ell$.
\end{lemma}

\begin{proof}
Let $U$ be a neighborhood of $(\epsilon_0,\Lambda_0)$ such that $\dim \hol(\Phi_j(\epsilon, \Lambda)) \leq \ell$ for all $(\epsilon,\Lambda) \in U$. The existence of $U$ is guaranteed by the upper semicontinuity property given in \cref{rem:trivial2} and the continuity of $\Phi_j$. Let $N$ be a submanifold of $A_j$ intersecting $U$. Using the rank theorem for Banach spaces as in \cite[Lemma 23]{dSLR15a} we conclude that $D\Phi_j(\epsilon,\Lambda)$ is injective on $T_{(\epsilon,\Lambda)}N$ for $(\epsilon,\Lambda)$ belonging to a dense open set in $N$.
By \cref{rem:trivial1} we obtain the following inequalities:
\begin{align*}
\dim N = \dim D\Phi_j(\epsilon,\Lambda)(T_{(\epsilon,\Lambda)} N) \leq \dim \hol(\Phi_j(\epsilon,\Lambda)) \leq \ell,
\end{align*}
which concludes the proof.
\end{proof}

\begin{proof}[Proof of \cref{thm:trivial}]
Define $\Lambda_0 = j_p^{\tz} H$. There exists $j\in J$ such that $(\epsilon_0,\Lambda_0) \in A_j$. By \cref{lem:dimEll} there exists a neighborhood $U$ of $(\epsilon_0, \Lambda_0)$ such that for any submanifold $N$ in $U\cap A_j$ it holds that $\dim N = 0$. Then the dimension of $U\cap A_j$ is zero, and thus $U\cap A_j$ consists of isolated points. By \cref{rem:trivial0} the proof is concluded.
\end{proof}


\subsection{CR geometry}\label{sec:CRgeometry}
In this subsection we briefly introduce some standard notation from CR geometry; more details can be found e.g.\ in \cite{BERbook}.
Let $M$ be a generic real-analytic CR submanifold of $\C^N$. It is well known (see \cite{BERbook}) that one can choose \textit{normal coordinates} $(z,w)\in \mathbb C^n\times \mathbb C^d=\mathbb C^N$ in such a way that $M$ is written as
\begin{align*}
w = Q(z,\bar z, \bar w), \quad \mathrm{(or~equivalently:} \ \bar w = \overline Q(\bar z,z, w) {\rm )},
\end{align*}
where $Q$ is a germ of a holomorphic map $Q:\mathbb C^{2n+d}\to\mathbb C^{d}$ satisfying $Q(z,0,\bar w) \equiv Q(0,\bar z,\bar w) \equiv \bar w$ and $Q(z, \bar z, \overline Q (\bar z, z, w)) \equiv w$.

In the proof of the parametrization results, the notion of \emph{Segre maps} is also needed. For $j\in \N$ let $(x_1,\ldots, x_j)$ be coordinates of $\mathbb C^{nj}$ ($x_\ell\in \mathbb C^n$ for $\ell = 1, \ldots, j$); we also write $x^{[j;k]} := (x_j,\ldots,x_k)$. The \textit{Segre map} of order $q\in \N$ is the map $S^q_0:\mathbb C^{nq}\to \mathbb C^N$ defined as follows:
\begin{align*}
S^1_0(x_1) := (x_1,0), \quad S^q_0\bigl(x^{[1;q]}\bigr) := \left (x_1, Q\left(x_1,\overline S^{q-1}_0\bigl(x^{[2;q]}\bigr) \right)\right).
\end{align*}
We say that  $M$ is \textit{minimal at $p \in M$} if it does not contain any germ of a CR submanifold $\widetilde M \subsetneq M$ of $\C^N$ through $p$ having the same CR dimension as $M$ at $p$.
The minimality criterion obtained in \cite{BER96} states that if $M$ is minimal at $0$, then $S_0^q$ is generically of full rank for sufficiently large $q$, and moreover, in this case, for every neighbourhood
$U \subset \C^{2qn} $ there exists $x^0 \in U$ sucht that
 $S_0^{2q} (0) = 0$ and such that $S_0^{2q}$ is of 
 full rank at $x^0$ (see e.g. \cite{BER99}).

\subsection{Real-analytic geometry}
\label{sec:realGeometry}
In order to prove our theorems we work within the framework of 
subanalytic and semi-analytic sets; to this end, we 
recall some basic notions and results.

A set $A\subset \mathbb R^n$ is called \emph{semi-analytic} if it is a finite union of intersections of sets defined by
 real-analytic equations and inequalities:
\begin{align}\label{eq:semi}
A = \bigcup_{i=1}^k \bigcap_{j=1}^{N(i)} A_{ij},
\end{align}
where $A_{ij}$ is either of the form $\{h_{ij}=0\}$ or $\{h_{ij}>0\}$ for some real-analytic $h_{ij}\in \mathbb R\{x_1,\ldots,x_n\}$.  The notion of semi-analytic set is modeled on the notion of semi-algebraic sets, which are defined in a similar way (with polynomial functions instead of analytic ones) and are closed under projections (Tarski-Seidenberg theorem). On the other hand, semi-analytic sets do not enjoy the same property, and therefore we will need some more subtle results.

If $\mathcal R$ is any ring of real functions over a set $E$, a subset $A\subset E$ is called \emph{definable over $\mathcal R$} if it can be expressed as \eqref{eq:semi}, with $A_{ij}$ being  either of the form $\{f_{ij}=0\}$ or $\{f_{ij}>0\}$ for some $f_{ij}\in \mathcal R$. We have the following (see e.g. \cite{BM88}):

\begin{theorem}[\textbf{{\L}ojasiewicz}]
\label{thm:Loj}
Let $X$ be an analytic manifold, let $A\subset X\times \mathbb R^k$ be definable over the ring $C^\omega(X)[x_1,\ldots,x_k]$ and let $\pi:X\times \mathbb R^k\to X$ be the projection on the first factor. Then $\pi(A)$ is semi-analytic, i.e., definable over $C^\omega(X)$.
\end{theorem}

The family of semi-analytic sets and maps (which are defined 
as maps whose graphs are semi-analytic sets) is not closed under projections. 
If one wants to work with images of semi-analytic sets, one encounters 
the subanalytic sets; for the basics, we refer the reader to \cite{BM88}.

A (subanalytic) \textit{cell decomposition} of a subanalytic set $A$ is a finite collection of subsets $\{C_j^q\}$ such that each $C_j^q$ is subanalytically homeomorphic to the ball $B^q = \{x\in \mathbb R^q : |x|<1\}$ ($C^q_j$ is then called a \emph{cell of dimension $q$}) and satisfies the following properties: 
\begin{enumerate}
\item $A= \bigcup_{j,q} C_j^q$
\item The closure $\overline C_j^q$ is the union of $C_j^q$ and cells of strictly smaller dimension.
\end{enumerate}
In this case way say that the sets $C_j^q$ form a \textit{stratification}.

We will need the following local triviality result of Hardt \cite{Hardt80} for bounded subanalytic maps. 

\begin{theorem}[\textbf{Hardt}]
\label{thm:Hardt}
Let $X$ and $Y$ be bounded subanalytic sets and $f:X\to Y$ be a continuous subanalytic map. Then there exist a finite subanalytic stratification $\{Y_1,\ldots,Y_k\}$ of $Y$, a collection $\{F_1,\ldots, F_k\}$ of bounded subanalytic sets, and subanalytic homeomorphisms $g_j\colon f^{-1}(Y_j)\to Y_j\times F_j$, such that $f|_{f^{-1}(Y_j)} = \pi \circ g_j$, where $\pi$ denotes the projection $\pi:Y_j\times F_j\to Y_j$.
\end{theorem}

\section{Basic properties of the mapping locus}
\label{sec:propLocus}

In this section we describe in more details some main properties of the mapping locus. In particular we are going to provide the proof of \cref{thm:semi-analytic}.

\begin{proof}[Proof of \cref{thm:semi-analytic}]
We follow the same steps as in \cite{dSJL15}. By \cite{Frisch67}, for any bounded semi-analytic set $B\subset X$ the ring $C^\omega(\overline B)$ is Noetherian. It follows that the ring $C^\omega(\overline B)[\Lambda]$ is also Noetherian.
To prove that $E_{\mathcal F}$ is locally semi-analytic for any $p\in E_{\mathcal F}$ we consider an open, bounded semi-analytic neighborhood $B$ of $p$. We will show that $E_{\mathcal F} \cap B$ is semi-analytic. 

Let $A$ be as in \eqref{defEquationJetParam}. Notice that by construction, it holds that $E_\mathcal F\cap B$ is precisely equal to the projection of $A\cap (B\times J_p^{\tz})$. Write now $A = \bigcup_{j\in J} A_j$, where $A_j = A \cap U_j$. Then since $E_{\mathcal F} \cap B$ is the union of the projections of $A_j\cap (B\times J_p^{\tz})$ and the finite union of semi-analytic sets is semi-analytic, it is enough to prove that each projection of $A_j\cap (B \times J_p^{\tz})$ is semi-analytic. In the following we fix $j\in J$ and omit the dependence on $j$ notationally.  Furthermore $c_i\in C^\omega(\overline B)$ for all $i\in \mathbb N$. Hence there exist finitely many indices $i_1,\ldots, i_\ell$ such that 
\[A\cap (\overline B\times J_p^{\tz})=\{c_{i_1}=c_{i_2}=\ldots=c_{i_\ell}=0\}.\]
It follows that $A\cap (\overline B\times J_p^{\tz})$ is definable in the sense of Lojasiewicz. By Lojasiewicz's Theorem (cf.\ \cite{dSJL15}), the projection of $A\cap (\overline B\times J_p^{\tz})$ onto $X$ is a semi-analytic subset of $X$. 
\end{proof}

The next result gives more information about maps $H_\epsilon: (M,p) \rightarrow (M'_\epsilon,p')$ with $\epsilon \in E_{\mathcal F}$. More specifically, such maps can be (at least generically) selected to depend on $\epsilon$ in an analytic way. The following result holds, cf. \cite{dSJL15}*{Lemma 8} for the equidimensional case:

\begin{lemma}\label{q0}
Suppose $E_{\mathcal F}$ contains a real-analytic submanifold  $R \subset  X$, and let $\epsilon \in R$ and fix a neighborhood $U \subset R$ of $\epsilon$. Then there exist $\delta_1 \in U$, a neighborhood $V$ of $\delta_1$, a real-analytic map $L: V \cap R \rightarrow J^{\tz}_0$ and maps $\mathcal F\ni H^{\delta}\colon (M,0) \to (M'_{\delta},0)$ such that $j^{\tz}_0 H^{\delta} = L(\delta)$ for all $\delta\in V$.
\end{lemma}

\begin{proof} Denote by $\pi \colon X \times J^{\tz}_0 \to X$ the projection
onto the first factor, and define $\mathcal R \coloneqq \pi^{-1} (R) \cap A$.
Then $\mathcal R$ is a semi-analytic subset of $X \times J^{\tz}_0$. Let
$\mathcal{ R}^{\rm reg} $ be the regular part of $\mathcal R$, which is an
open dense smooth semi-analytic subset of $\mathcal R$. For any point $a \in
\mathcal{ R}^{\rm reg} $ define $r(a) \coloneqq \rank(\pi|_{\mathcal{ R}^{\rm
reg} }(a))$ and let $r_0 \coloneqq \max_{a\in \mathcal{ R}^{\rm reg} } r(a)$.
Define $\tilde{\mathcal{ R}} \coloneqq \{a\in \mathcal{ R}^{\rm reg} : r(a) =
r_0\}$, then $\tilde{\mathcal{ R}}$ is an open dense subset of $\mathcal{
R}^{\rm reg} $, and thus of $\mathcal R$. Since $\pi (\mathcal{R}) = R$ 
by assumption, we note that 
it holds that $\tilde{R} \coloneqq \pi(\tilde{\mathcal{
R}})$ is open and dense in $R$. Let $\delta_1 \in U \cap \tilde{R}$ and let $a_1 \in
\tilde{\mathcal{ R}}$ such that $\pi(a_1) = \delta_1$, i.e.
$a_1=(\delta_1,\Lambda_1)$ for a certain $\Lambda_1 \in J^{\tz}_0$. By the
constant rank theorem there is a neighborhood $\mathcal V$ of $a_1\in
\tilde{\mathcal{ R}}$, a neighborhood $V$ of $\delta_1$ in $\tilde{R}$, a ball $B$ in
some $\R^m$ and an analytic diffeomorphism $\psi: \mathcal V \rightarrow V
\times B$, such that $\psi(a_1) = (\delta_1,0)$, and such that $\pi \circ \psi = \pi$.
Define $N \coloneqq V
\times \{0\} \subset V \times B$ and consider $L \coloneqq \pi_2 \circ
\psi^{-1}|_N$, where $\pi_2$ is the projection on the second factor, the jet
space. Then $L$ is a real-analytic map 
and the proof is concluded by setting $H^{\delta} \coloneqq
\Phi(\delta,L(\delta))$ for all $\delta\in V$.
\end{proof}

The previous selection lemma can be used to generically estimate the dimension of the mapping locus by means of the space of infinitesimal deformations, cf. \cite{dSJL15}*{Theorem 6} for the equidimensional case.

\begin{theorem}
\label{thm:tangentSpaceInfDef}
Let $(M'_{\epsilon}, 0)$, $\epsilon \in X$ be a deformation of $(M', 0)$ as before. Let $S$ be a semi-analytic subset of $E_{\mathcal F}$ and $S^{\rm reg}$ be the set of regular points in $S$. Then there exists an open dense subset $D$ of $S^{\rm reg}$ with the following properties:
\begin{itemize}
\item[(i)] For every $\epsilon \in D$ there exists a neighborhood $U$ of $\epsilon$ in $S^{\rm reg}$ and an analytic map $\phi: U \rightarrow \mathcal F$ such that $\phi(\delta): (M,0) \rightarrow (M'_{\delta},0)$ (note that $U$ as an open subset of $S^{\rm reg}$ is an analytic submanifold).
\item[(ii)] For all $\delta\in D$ we have:
\begin{align*}
T_{\delta} S^{\rm reg} \subset \pi_1(\hol(\phi(\delta),\mathfrak{D})).
\end{align*}
\item[(iii)] If $M'_{\epsilon}$ is a base-point-type deformation (in particular $X=M'$), then for all $\delta\in D$ we have 
\begin{align*}
T_{\delta} S^{\rm reg} \subset \hol(\psi(\delta))(0)
\end{align*}
where we define $\psi(\delta)=\phi(\delta)+\delta$.
\end{itemize}
\end{theorem}
\begin{proof}
Note that since $S$ is semi-analytic, by \cite{Lojasiewicz65, BM88} $S^{\rm reg}$ is a dense semi-analytic subset of $S$. Let $D$ be the set of the points satisfying (i) and (ii): We will show that $D$ is an open, dense subset of $S^{\rm reg}$.  Let $\epsilon_0\in S^{\rm reg}$ and let $O$ be a neighborhood of $\epsilon_0$, and let $\delta_1\in X$ be the point given by \cref{q0} (with $R=O\cap S^{\rm reg}$): we will show that $\delta_1\in D$. Indeed, take $U=V$, where $V$ is given by \cref{q0}. To show that property (i) is satisfied, define $\phi$ as $\phi(\delta)=H^{\delta}$ for all $\delta\in U$, where $H^\delta$ is the one from \cref{q0}.
To establish (ii) let $v\in T_{\delta_1} S^{\rm reg}$. Let $\delta(t)$ be a smooth curve in (a neighborhood of $\delta_1$ in) $D$ such that $\delta(0)=\delta_1$ and $\delta'(0)=v$. We define a smooth curve of maps $c_t:(M,0)\to (M'_\delta(t),0)$ as $c_t=\phi(\delta(t))$. Let $Y$ be the derivative $\frac{dc_t}{dt}|_{t=0}$ of $c_t$ at $t=0$, then $(v,Y) \in \hol(c_0,\mathfrak{D}) =\hol(\phi(\delta(0)),\mathfrak{D})=\hol(\phi(\delta_1),\mathfrak{D})$ (cf. \cref{rem:derivDef}). Thus $v = \pi_1(v,Y) \in \pi (\hol(\phi(\delta_1),\mathfrak{D}))$. Since $v$ is an arbitrary vector of $T_{\delta_1}S^{\rm reg}$, we conclude that $T_{\delta_1} S^{\rm reg} \subset \pi_1(\hol(\psi(\delta_1),\mathfrak{D}))$.
In the case of the base-point-type deformation, (iii) is a direct consequence of \cref{cor:differentHol}.
This shows that $\delta_1 \in D$.
Repeating the same arguments for any $\delta \in U$ shows that $\delta \in D$, hence $D$ is an open subset of $S^{\rm reg}$.
\end{proof}

It is worth noting that a stronger selection lemma than in \cref{q0} can be proved using more advanced results of real-analytic geometry. The following lemma is the analogous of \cite[Lemma 6]{dSJL15}.

\begin{lemma}
Let $\epsilon\in E_{\mathcal F}$, and fix a neighborhood $U$ of $\epsilon$. Then there exists $\delta_0\in E_{\mathcal F}\cap U$ such that for every map $\mathcal F\ni H^{\delta_0}\colon (M,0) \to (M'_{\delta_0},0)$ and every neighborhood $W$ of $j_0^{\tz} H^{\delta_0}$  there exists a neighborhood $V$ of $\delta_0$  such that the following holds: For all  $\delta\in V\cap E_{\mathcal F}$ there exists a map $\mathcal F\ni H^{\delta}:(M,0)\to (M'_\delta,0)$ such that $j_0^{\tz} H^\delta \in W$.
\end{lemma}

\begin{proof}
Let $A$ be as in \eqref{defEquationJetParam}. It is enough to show the conclusion for each $A_j$, which we denote by $A$ by an abuse of notation. Just as in the proof of \cref{thm:semi-analytic}, for a small ball $B\subset U$ we can write $A\cap (\overline B\times J_p^{\tz})$ as the intersection of the vanishing set of finitely many $c_i$. Using the same argument as Hardt (see \cite{Hardt80}, Step I of the proof of the main theorem) we can reduce to the case of bounded subanalytic sets, by a fibrewise projectivization of $A$; this is possible since $A$ is defined by functions which are polynomial in the jet variable. 
By the version of Hardt's theorem for subanalytic sets (see \cite{Hardt80}, in particular the remarks starting at the end of page 291), there exists a partition of $\pi(A\cap B)$  into subanalytic sets $C_1,\ldots,C_m$  in such a way that $\pi|_{\pi^{-1}(C_j)}$ is trivial for $1\leq j\leq m$. Furthermore we can find a stratification $C^i_j$ of $\pi(A\cap B)$ by smooth subanalytic sets respecting $\{C_1,\ldots,C_m\}$  (i.e.\ $\overline C^i_j$ is the union of $C^i_j$ and strata of strictly smaller dimension): see \cite{Lojasiewicz65}. Let $d= \max\{d':\exists \delta\in E_{\mathcal F}\cap B \ {\rm s.t.}\ \delta \in C^i_j\ {\rm and}\ \dim C^i_j = d'  \}$. Let $\delta_0\in  E_{\mathcal F}\cap B$ such that $\delta_0\in C^{i_0}_{j_0}$ with $\dim C^{i_0}_{j_0}=d$. Then, by the stratification property, there exists a neighborhood $V'$ of $\delta_0$ such that $E_{\mathcal F}\cap V'\subset C^{i_0}_{j_0}$ (see for instance \cite{dSJL15}*{proof of Lemma 6}).
By Hardt's theorem there exist a subanalytic set $Y \subset J_0^{\tz}$ and a subanalytic homeomorphism $\psi:\pi^{-1}(C^{i_0}_{j_0}) \rightarrow C^{i_0}_{j_0} \times Y$ such that $\pi \circ \psi = \pi$. 
Let $H^{\delta_0}$ be a map $(M,0) \rightarrow (M'_{\delta_0},0)$ and $W$ a neighborhood of $j_0^{\tz} H^{\delta_0}$. Define $O \coloneqq \psi((V' \times W) \cap \pi^{-1}(C^{i_0}_{j_0}))$, which is an open neighborhood of $(\delta_0, j_0^{\tz} H^{\delta_0})$ in $C^{i_0}_{j_0} \times Y$. Then there is an open set $V \subset V'$ such that for every $\delta \in V$ it holds that $(\delta,j_0^{\tz} H^{\delta_0}) \in O$.
Then $(\delta,j_\delta) \coloneqq \psi^{-1}(\delta,j_0^{\tz} H^{\delta_0})$ has the property that $j_\delta \in W$ and setting $H^{\delta} \coloneqq \Phi(\delta,j_\delta)$, where $\Phi$ is from \eqref{rationalJetParam}, shows the claim.
\end{proof}

\section{A class of maps satisfying JPP}
\label{section:mapsJPP}

We are now going to define a class of maps satisfying JPP from \cref{def:jetparamdef}, namely the \textit{finitely nondegenerate maps} (\cite{Lamel01}). The following definition is taken almost verbatim from \cite{dSLR17}; since some adaptations to the setting of this paper are needed we state the definition again for the convenience of the reader.

\begin{definition}\label{defNondeg}
Let $M'$ be a real-analytic subvariety of $\C^{N'}$, $p' \in M'$ and $\varrho_1,\ldots, \varrho_d$ generators of $\mathcal I_{p'}(M')$. Let $(L_1,\ldots, L_n)$ be a basis of CR vector fields of $(M,p)$. For a multiindex $\alpha =(\alpha_1,\ldots,\alpha_n)\in \N^n$ we write $L^{\alpha} = L_1^{\alpha_1} \cdots L_n^{\alpha_n}$. Given a holomorphic map $H=(H_1,\ldots,H_{N'})\in \mathcal H ((M,p),(M',p'))$ with $H(p)=p'$, 
and a fixed sequence $\iota = (\iota_1,\ldots,\iota_{N'})$ of 
 multiindices $\iota_m\in\N_0^n$ and $N'$-tuple of integers $\ell = (\ell^1,\ldots, \ell^{N'})$ with $1 \leq \ell^k \leq d'$, we consider the determinant
 \begin{equation} \label{folcon}
s^{\iota,\ell}_H(Z) = \det \left(\left(L^{\iota_j}\varrho_{\ell^j, Z_k'}(H(Z),\overline H(\bar Z)) \right)_{1\leq j,k \leq N'} \right).
\end{equation}

We define the open set $\mathcal F_{k}((M,p),(M',p')) \subset \mathcal H((M,p),(M',p'))$ as the set of maps $H$  for which there exists such a sequence of multiindices $\iota = (\iota_1,\ldots,\iota_{N'})$ satisfying $k = \max_{1 \leq m \leq N'}|\iota_m|$ and $N'$-tuple of integers $\ell = (\ell^1,\ldots, \ell^{N'})$ as above  such that $s_H^{\iota,\ell} (p)\neq 0$. We define $J_{k_0}$ as the set of all pairs $j=(\iota,\ell)$, where  $\iota = (\iota_1,\ldots,\iota_{N'})$ is a sequence of multiindices  with $k_0 = \max_{1\leq m \leq N'} |\iota_m|$ and $\ell = (\ell^1,\ldots, \ell^{N'})$ is as above.
We will say that $H$ with $H(M) \subset M'$ is {\em $k_0$-nondegenerate} at $p$ if $k_0 = \min \{ k\colon H \in \mathcal{F}_k((M,p),(M',p')) \}$ is a finite number. 
We write $\mathcal{F}_{k_0}((M,p),(M',p'))$ for the (open) 
subset of $\mathcal{H} ((M,p),(M',p'))$ containing all $k_0$-nondegenerate maps. A holomorphic map $H$ with $H(M) \subset M'$ is called {\em $\ell$-finitely nondegenerate} if for each $p\in M$ the map $H$ is $k(p)$-nondegenerate at $p$ and $\ell = \max\{k(p): p \in M\}$.
\end{definition} 

Let us stress that the distinction between $k$-nondegeneracy and 
$\ell$-finite nondegeneracy is important: A map is $\ell$-finitely nondegenerate
if and only if it is $k$-nondegenerate for some $k\leq \ell$. Our next
two results show that both $k$-nondegenerate and $\ell$-finitely nondegenerate
maps satisfy the JPP. 

\begin{theorem}\label{jetparamdef} 
Let $(M,p) \subset \CN$ be a germ of a generic minimal real-analytic submanifold of $(\C^N,p)$ and $(M'_{\epsilon},p')_{\epsilon \in X}$ be a deformation. Fix $k_0 \in \N$ and let $\mathbf t$ be the minimum integer, such that the Segre map $S^{\mathbf t}_p$ of order $\mathbf t$ associated to $M$  is generically of full rank. Then $\mathcal F_{k_0,\epsilon} \coloneqq \mathcal F_{k_0}((M,p),(M'_{\epsilon},p'))$ satisfies \cref{def:jetparamdef} with $\tz = 2\mathbf t k_0$. \end{theorem}

The following corollary is immediate from the definition of the JPP and 
the preceding theorem. 

\begin{corollary}
\label{jetparam2} Let $(M,p) \subset \CN$ be a germ of a generic minimal real-analytic submanifold of $(\C^N,p)$ and $(M'_{\epsilon},p')_{\epsilon \in X}$ be a deformation. Fix $\ell \in \N$ and let $\mathbf t$ be the minimum integer, such that the Segre map $S^{\mathbf t}_p$ of order $\mathbf t$ associated to $M$  is generically of full rank. 
Then the family $(\mathcal{F}^{\ell}_{\epsilon})_{\epsilon\in X}$, 
where $\mathcal{F}^\ell_{\epsilon} = \bigcup_{k=0}^\ell \mathcal F_{k,\epsilon}$,  satisfies \cref{def:jetparamdef} with $\tz = 2\mathbf t \ell$.
\end{corollary}

In order to show \cref{jetparamdef} we follow the line of thought as in \cite{dSLR17}. The next Lemma is an immediate consequence of \cite[Prop. 25, Cor. 26]{Lamel01} and of standard parametrization techniques. The key fact we need to use is that in the basic identity of \cite[Prop. 25]{Lamel01} the map $\Psi$ will depend analytically on $\epsilon$, since the $\Phi_j$ appearing in the proof depend polynomially on finitely many derivatives of the defining function of $M'$. Furthermore the implicit function theorem used to obtain $\Psi$ from the $\Phi_j$ preserves analyticity in $\epsilon$; using \cite{dSLR17}*{Prop. 37} allows to prove the following.

\begin{lem}\label{lem:derivBasicIdentitydef}
Under the assumptions of \cref{jetparamdef} the following holds: For all $\ell \in \N$ and $j \in J_{k_0}$ there exists a holomorphic mapping $\Psi^j_{\ell}: X \times \C^N \times \C^N \times \C^{K(k_0 + \ell) N'} \to \C^{N'}$ such that for every curve $(\epsilon(t),H(t)) \in \parcur{r}_{(\epsilon,H)}$ with $H\in \mathcal F_{k_0,\epsilon}$ such that $s^j_H(p)\neq 0$, where $s^j_H$ is given as in \eqref{folcon}, we have for sufficiently small $t$
\begin{align}
\label{derivBasicIdentitydef}
\partial^{\ell} H(Z,t) = \Psi^j_{\ell}(\epsilon(t),Z,\zeta, \partial^{k_0+ \ell} \bar H(\zeta,t)) + O(t^{r+1}), 
\end{align}
for $(Z,\zeta)$ in a neighborhood of $(p,\bar p)$ in $\mathcal M$, where $\partial^{\ell}$ denotes the collection of all derivatives up to order $\ell$. Furthermore there exist polynomials $P^{\ell,j}_{\alpha, \beta},Q_{\ell,j}$ and integers $e^{\ell,j}_{\alpha,\beta}$ such that 
\begin{align}
\label{derivBasicIdentityRational}
\Psi_{\ell,j}(\epsilon,Z,\zeta,W) = \sum_{\alpha,\beta \in \N_0^N} \frac{P^{\ell,j}_{\alpha,\beta}(\epsilon,W)}{Q_{\ell,j}^{e^{\ell,j}_{\alpha,\beta}}(\epsilon,W)} Z^{\alpha} \zeta^{\beta}.
\end{align}

\end{lem}

Evaluating \eqref{derivBasicIdentitydef} along Segre sets we obtain the following result. Notice that the Segre sets involved do not depend on $\epsilon$, so the analytic dependence on $\epsilon$ is preserved, when equations of the form \eqref{derivBasicIdentitydef} are evaluated along the Segre sets.

\begin{cor}\label{cor:iterationSegredef}
For fixed $j\in J, q \in \N$ with $q$ even there exists a holomorphic mapping $\varphi^j_{q}: X \times \C^{qn} \times \C^{K(q k_0) N'}\to \C^{N'}$ such that for every curve $(\epsilon(t),H(t)) \in \parcur{r}_{(\epsilon,H)}$ with $H \in \mathcal F_{k_0,\epsilon}$ such that $s^j_H(p)\neq 0$, where $s^j_H$ is given as in \eqref{folcon}, we have for sufficiently small $t$
\begin{align}
\label{iterationSegredef}
H(S^q_p(x^{[1;q]}),t) = \varphi_{q}^j(\epsilon(t), x^{[1;q]}, j_p^{q k_0} H(t))+O(t^{r+1}).
\end{align}
Furthermore there exist (holomorphic) polynomials $R^{q,j}_{\gamma},S_{q,j}$ and integers $m^{q,j}_{\gamma}$ such that 
\begin{equation}
\label{iterationSegreRationaldef}
\varphi^j_{q}(\epsilon,x^{[1;q]},\Lambda) =  \sum_{\gamma \in \N_0^{qn}} \frac{R^{q,j}_{\gamma}(\epsilon,\Lambda)}{S_{q,j}^{m^{q,j}_{\gamma}}(\epsilon,\Lambda)} (x^{[1;q]})^{\gamma}.
\end{equation}
\end{cor}

\begin{proof}[Proof of \cref{jetparamdef}] The following is an adaptation of the proof as given in \cite{dSLR17}*{proof of Thm. 36}. By the choice of $\mathbf t\leq d+1$, the Segre map $S^{\mathbf t}_p$ is generically of maximal rank. By Lemma 4.1.3 in \cite{BER99}, the Segre map $S^{2\mathbf t}_p$ is of maximal rank at $p$. Using the constant rank theorem, there exists a neighborhood $\mathcal V$ of $S^{2\mathbf t}_p$ in $(\mathbb C\{x^{[1;2\mathbf t]}\})^N$ and a map $T:\mathcal V\to (\C\{Z\})^{2\mathbf t n}$ such that $A\circ T(A)=Id$ for all $A\in \mathcal V$.
We now define the holomorphic map
\begin{equation}\label{defphi}\phi:\mathcal V\times (\mathbb C\{x^{[1;2\mathbf t]}\})^{N'} \to (\mathbb C\{Z\})^{N'}, \quad \phi(A,\psi)=\psi(T(A)).
\end{equation}
 Thus we have  that $\phi(A,h\circ A) = h(A(T(A)))=h$ for all $A\in \mathcal V$ for all $h\in(\mathbb C\{Z\})^{N'}$.  
We define $\Phi_j(\epsilon,\cdot, \Lambda)=\phi(S_p^{2\mathbf t},\varphi^j_{{2\mathbf t}}(\epsilon,\cdot, \Lambda))$. Note that $\Phi_j$ depends analytically on $\epsilon$, applying $\phi(S_p^{2\mathbf t},\cdot)$ to both sides of equation \eqref{iterationSegredef} with $q=2\mathbf t$ we get
\begin{align*}
H(t) & =  \phi(S_0^{2\mathbf t},H(t)\circ S^{2\mathbf t}_0)=\phi(S_0^{2\mathbf t},\varphi^j_{{2\mathbf t}}(\epsilon(t),\cdot, j_0^{{2\mathbf t k_0}} H(t)) + O(t^{r+1})) \\
 & = \Phi_j(\epsilon(t),\cdot,  j_0^{{2\mathbf t k_0}} H(t)))+O(t^{r+1}) .
\end{align*}
which gives (b) in JPP. 
By setting $q_{j}(\epsilon,\Lambda) = S_{2\mathbf t,j}(\epsilon,\Lambda)$, where $S_{2\mathbf t,j}$ is given in \eqref{iterationSegreRationaldef}, a direct computation using \eqref{iterationSegreRationaldef} and \eqref{defphi} allows to derive the expansion in \eqref{rationalJetParam}.
Let $\mathcal U_{j}$ be a neighborhood of $\{p\}\times U_{j}$ in $\mathbb C^N$ such that $\Phi_{j}$ is convergent on $\mathcal U_{j}$. By applying the usual procedure of plugging the form \eqref{rationalJetParam} into the mapping equation (after choosing a parametrization $t\mapsto \Sigma(t)$ of $M$) and developing in powers of $t$ we obtain \eqref{defEquationJetParam}. The $c_i^j$ appear as coefficients of the powers of $t$ in the expansion and depend analytically on $\epsilon$, because $\Phi_j$ and every generator of $\mathcal I_{p'}(M'_\epsilon)$ do. The remaining fact about $c^j_i$ follows in the same way as in \cite{dSLR17}*{proof of Thm. 36}.
\end{proof}

\section{Construction of a singular mapping locus}
\label{sec:singularlocus}

In this section we are going to provide an example of a mapping locus whose irreducible components are singular. This phenomenon cannot happen in the equidimensional case and highlights one difference with the positive codimensional case.

Let $N,N'\in \mathbb N$ with $2N>N'$. We write coordinates on $\mathbb C^{N+1}$ as $(Z,w)$ with $Z=(z_1,z_2,\ldots, z_N)$, and on $\mathbb C^{N'+1}$ as $(Z',w')$ with $Z'=(z_1',z_2',\ldots, z_{N'}')$. For any $j,k\in \mathbb N$, we denote by $\mathfrak N(j,k)$ the number of monomials in $j$ variables of order less than or equal to $k$.

Let $M=\{{\rm Im}\, w=|z_1|^2+ |z_2|^2+\ldots + |z_N|^2\}\subset \mathbb C^{N+1}$. Let $\mathfrak J^{k,d}_0$ be the space of $k$-jets at $0$ of real maps $\rho':\mathbb C^{N'+1}\to \mathbb R^d$ such that $\rho'(0)=0$, and let $\mathcal J^k_0$ be the space of $k$-jets of holomorphic maps $\psi:\mathbb C^{N'+1}\to\mathbb C^{N'+1}$ such that $\psi(0)=0$. 
We can define a real-algebraic action of $\mathcal J^k_0$ on $\mathfrak J^{k,d}_0$ in the natural way:
if $\psi\in \mathcal J^k_0$ and $\rho'\in \mathfrak J^{k,d}_0$, then $\psi\cdot \rho'=j^k_0(\rho'\circ\psi )$.

\begin{lemma}
\label{lem:semiB}
There exists $k_0\in \mathbb N$ and a semi-algebraic subset $B\subset \mathfrak J_0^{k_0,d}$ of positive codimension such that if $\rho': \mathbb C^{N'+1}\to \mathbb R^d$ is a local defining function such that $M$ admits a holomorphic embedding in $\{\rho'=0\}$ passing trough $0$ then $j^{k_0}_0 \rho'\in B$.
\end{lemma}

\noindent (With an abuse of notation, we will identify with $B$ the analogous set $B_p$ contained in the space $\mathfrak J^{k_0, d}_p$ of jets about $p\neq 0$.)

\begin{proof}: For any $k\in \mathbb N$, consider the subset $B'_k\subset \mathfrak J_0^{k,d}$ consisting of the $k$-jets of functions $\rho':\mathbb C^{N'+1}\to\mathbb R^d$ satisfying the relation
\begin{equation}\label{straight}
\rho'(Z,  0 , u+i\|Z\|^2)=O(k+1).
\end{equation}
For any multi-indices $J\in \mathbb N^{2N+1}$ and $K\in \mathbb N^{2N'+1}$, where $J=(j,\overline j,\ell)$ with  $j=(j_1,\ldots, j_N)$, $\overline j=(\overline j_1, \ldots, \overline j_N)$ and $K=(k,\overline k,m)$ with $k=(k_1,\ldots, k_{N'})$, $\overline k=(\overline k_1, \ldots, \overline k_{N'})$ let
 \[D_J=\frac{\partial^{|J|}}{\partial Z^j \partial \overline Z^{\overline j} \partial u^\ell}, \ \ D'_K=\frac{\partial^{|K|}}{\partial {Z'}^k \partial \overline {Z'}^{\overline k} \partial {u'}^m}.\]
  Applying the $D_J$ derivative (with $|J|\leq k$) to the left hand side of (\ref{straight}), using the chain rule and evaluating at $0$ we obtain 
 \begin{equation}\label{DprimeJ}
 D'_J\rho'(0) + \mbox{(algebraic expression in lower order derivatives)} = 0
 \end{equation}
 The system of equations in $\mathfrak J^{k,d}_0$ given by \eqref{DprimeJ} (with $|J|\leq k$) has a Jacobian of full rank at $0$, as can be easily verified by taking in account the triangular structure of the system.  Hence $B'_k$ is a smooth algebraic submanifold of $\mathfrak J^{k,d}_0$, of codimension $d \cdot \mathfrak N(2N+1,k)$.

Note that the image $H(M)$ of any holomorphic embedding of $M$ into $\mathbb C^{N'+1}$ can be turned into the manifold
\begin{equation}\label{N+1plane}
\{z_{N+1}'=\ldots = z'_{N'} =0,\ {\rm Im}\, w'=|z_1'|^2 +\ldots + |z'_N|^2\}
\end{equation}
by a local change of coordinates. Indeed, if $H:M\to H(M)$ is such an embedding, we can first apply a change of coordinates in $\mathbb C^{N'+1}$ which sends $H(\mathbb C^{N+1})$ to the complex $(N+1)$-plane $\{z_{N+1}'=\ldots = z'_{N'} =0\}$. After this, we can apply a new coordinate change involving only $z_1',\ldots, z_N'$ which sends $H(M)$ to the manifold described by \eqref{N+1plane}.

Thus, for any $k\in \mathbb N$ the set $B_k:=\mathcal J^k_0\cdot B'_k$ (i.e.\ the orbit of $B'_k$ by the action of $\mathcal J^k_0$) contains all the $k$-jets of local defining functions $\rho'$ such that $\{\rho'=0\}$ admits a local embedding of $M$. Moreover, $B_k$ is a semi-algebraic subset of $\mathfrak J^{k,d}_0$ because it is the image of a (real)-algebraic map defined on (real)-algebraic manifolds.  Since the dimension of $\mathcal J^k_0$ is $2\mathfrak N(N'+1,k)$, the dimension of $B_k$ is at most $\dim(B'_k)+2\mathfrak N(N'+1,k)$, and its codimension is at least $d\mathfrak N(2N+1,k)-2\mathfrak N(N'+1,k)$. Since $2N>N'$, there exists $k_0$ such that $d\mathfrak N(2N+1,k_0)-2\mathfrak N(N'+1,k_0)>0$. The proof of the lemma is concluded by putting $B=B_{k_0}$.
\end{proof}

\begin{proof}[Proof of \cref{thm:singularLocus}]
Write $n = N - 1$. Let $u,Z=(z_1,\ldots, z_n),\tau$ be coordinates in $\mathbb R\times \mathbb C^n\times \mathbb C$ with $Z=(z_1,\ldots, z_n)$, $z_j=x_j+i y_j$, $\tau=s+it$. Let $Y\subset \mathbb C$ be a singular real-analytic subset defined as $Y=\{r(s,t)=0\}$ with $r$ vanishing at $0$.
For certain (real) polynomial functions $\alpha_1,\ldots, \alpha_n,\beta,\gamma:\mathbb R\times \mathbb C^n\times \mathbb C\to \mathbb C$, to be determined later, we define the following map $\phi:\mathbb R\times \mathbb C^n\times \mathbb C\to \mathbb C^{N+1}$:
\[\phi(u,Z,\tau)=(z_1,\ldots, z_n, z_1^2+\ldots+z_n^2+ \tau, u+i|Z|^2)  + r(\tau, \bar \tau)  \left(  \alpha_1, \ldots, \alpha_n,\beta,\gamma \right) .\]
Then $d\phi$ has (real) rank $2 n + 3$ (at least around the origin) and $M'=\phi(\mathbb R\times \mathbb C^n\times \mathbb C)\subset \mathbb C^{N+1}$ is a real-analytic hypersurface (at least around the origin). We denote by $\rho'$ the (uniquely determined) defining function of $M'$ of the form $\rho'={\rm Im}\, w' - f(Z', {\rm Re}\, w')$.

For any $\tau_0\in Y$, the map $\psi_{\tau_0}:\mathbb C^N\to\mathbb C^{N+1}$ defined as
\[\psi_{\tau_0}(Z,w)=(z_1,\ldots,z_n,z_1^2+\ldots+z_n^2+\tau_0, w)\]
is an embedding of $M$ into $M'$. Define
\begin{align*}
C \coloneqq \bigcup_{\tau_0\in Y} \psi_{\tau_0}(M).
\end{align*}
Note that $C$ is a real-analytic subset of $M'$ of dimension $2N$ and $C\cong Y \times M$.
Let $E\subset M'$ be the mapping locus, i.e. the set of the points of $M'$ admitting a local holomorphic mapping of $M$. By construction and the homogeneity of $M$ it holds that $C \subset E$. We want to show that for a suitable choice of $\alpha_1,\ldots, \alpha_n,\beta,\gamma$ the set $E$ is contained in a proper semi-analytic subset $\widetilde E$ of $M'$ of dimension $2N$.
This proves the theorem, since if $C\subset E\subset \widetilde E$, where the dimension of $\widetilde E$ is equal to the dimension of $C$, then $E$ is a singular semi-analytic subset having $C$ as an irreducible component.

Let $\widetilde E$ be the set of the $p\in M'$ such that $j^{k_0}_p\rho'\in B$, where $B$ is given in \cref{lem:semiB} with $d=1$ and $N'=N$; then $\widetilde E$ is a semi-analytic subset of $M'$, as a preimage of the semi-algebraic set $B$ under the analytic map $p\mapsto j^{k_0}_p\rho'$. Moreover, by definition $E\subset \widetilde E$. 
Note that, since $B$ is semi-algebraic and of positive codimension, it is contained in the union of finitely many proper algebraic subvarieties $B_1,\ldots, B_\ell$ (this can be seen by disregarding the inequalities in the definition of $B$). To show that $\widetilde E$ is a semi-analytic subset of positive codimension, then it will be enough to check that for a suitable choice of $\alpha_1,\ldots,\alpha_n, \beta,\gamma$ and of $\tau_0$ close enough to $0$, putting  $\phi(0,\tau_0)=p_0$ we have $j^{k_0}_{p_0}\rho' \not\in B_j$ for all $j=1,\ldots, \ell$. In this case $\widetilde E$ is contained in the union of the real-analytic sets $X_j=\{p\in M': j^{k_0}_p\rho'\in B_j\}$, each of them proper since $p_0\not\in X_j$, and hence it has positive codimension.

To show the claim, choose $\tau_0=s_0+it_0$ close enough to $0$ such that $\tau_0\not\in Y$, consider the change of coordinates $s'=s-s_0,t'=t-t_0$, and let $q$ be such that $q(s',t')=r(s,t)$. We have $q(0,0)\neq 0$, so that we can consider the real power series centered at $0$ defining $1/q$. Let $\delta$ be any complex-valued power series in the variables $u,Z,\tau'=s'+it'$, and let $\alpha=j^{k_0} (\delta \frac{1}{q})$: then $\alpha$ is a polynomial in $u,Z,\tau'$ and 
\[j^{k_0} (q \alpha)= j^{k_0}\left(j^{k_0} q j^{k_0}\left(\delta \frac{1}{q}\right)\right) = j^{k_0}\left(j^{k_0}(\delta)j^{k_0}(q)j^{k_0}\left(\frac{1}{q}\right)\right) = j^{k_0} \delta.\]
In other words we can obtain any prescribed jet of $\phi$ at $(0,\tau_0)$, and hence any prescribed jet of $\rho'$ at $\phi(0,\tau_0)=p_0$, by the appropriate choice of $\alpha_1,\ldots,\alpha_n, \beta,\gamma$: in particular we can choose them in such a way that $j^{k_0}_{p_0}\rho'\not\in \bigcup_{j=1}^\ell B_j$.
\end{proof}

\begin{remark}
More specifically, the proof of \cref{thm:singularLocus} shows that any given singular set $Y$ of $\C$ can be realized in the following sense: We can construct a hypersurface $M' \subset \C^{N+1}$ such that there exists an irreducible component $C$ of $E_{\mathcal F}$ with the property that $C\cong Y \times M$.
\end{remark}

\section{Examples}
\label{sec:examples}

We wish to present a list of examples showing the various phenomena that can appear in the geometric structure of the mapping locus. Similar behaviors can be seen in the (equidimensional) case of the equivalence locus, cf. \cite[\S 7]{dSJL15}. Indeed that one is a special case of our setting. However we would like to construct examples in the positive codimension case.

\begin{example}
\label{ex:example1}
Let $(M,0)$ be the germ at $0$ of $M=\{\im w = |z|^2\} \subset \C^2$, and $M' = \{\im w' = |z_1'|^4 + |z_2'|^4\}$. In this case we take $\mathcal F = \mathcal F_k$ as the set of $k$-nondegenerate maps from $M$ into $M'$. Then the mapping locus $E_{\mathcal F}$ is equal to $M' \setminus Y$, where $Y=(\{z_1=0\} \cup \{z_2=0\})\cap M'$. Indeed, there are no $k$-nondegenerate maps $H$ from $M$ into $M'$, such that $H(0) \in Y$. This follows from the fact that $M'$ is not finitely nondegenerate at the points of $Y$ and that image of $k$-nondegenerate maps is contained in the set of points of $M'$ at which $M'$ is at most $k$-nondegenerate, cf. \cite[\S 3]{Lamel01}.
Moreover $M'\setminus Y \subset E_{\mathcal F}$: Let $p' = (\hat z_1',\hat z_2',\hat w') \in M'\setminus Y$ and $\psi(z_1',z_2',w') = ({z_1'}^2,{z_2'}^2,w')$. Denote $S^3 = \{\im w' = |z_1'|^2+ |z_2'|^2 \}$. Since $p' \not\in Y$ there is a neighborhood $U'$ of $p'$ such that $\psi: U' \rightarrow V' = \psi(U')$ is a biholomorphism and $\psi(U'\cap M') = S^3\cap V'$. A map $H: M \rightarrow M'$ with $H(0) = p'$ can then be constructed by choosing a $2$-nondegenerate map $F$ from $M$ into $S^3$ such that $F(0) = \psi(p')$ and taking $H = \psi^{-1} \circ F$.

This example shows that the mapping locus need not be an analytic variety, but in general it is just semi-analytic, that is, inequalities can indeed occur.
\end{example}

\begin{example}
\label{ex:example2}
Let $(M,0)$ be the germ at $0$ of $M=\{\im w = |z|^2 + |z|^4\} \subset \C^2$ and $M'=\{\im w' = |z_1'|^2 + |z_1'|^4 + (\re(z_1'))^2 \im(z_2') \} \subset \C^3$. Let $\mathcal F = \mathcal F_2$. The map $H_{s,t}: (z,w)\mapsto (z,t,w+s)$ sends $(M,0)$ into $M'$ for all $(s,t) \in \R^2$ and belongs to $\mathcal F$, hence $S= \{(0,t,s): (s,t)\in \R^2\}\subset E_{\mathcal F}$ and $T_p S \subset \hol(H_{s,t})(0)$, where $p=(0,t,s)$, cf. \cref{thm:tangentSpaceInfDef}. Indeed, a computation shows that $T_p S = \hol(H_{s,t})(0)$. 
\end{example}

\begin{example}
\label{ex:example3}
Let $M, M', H_{0,t}$ and $\mathcal F$ be as in \cref{ex:example2}. For a closed subset $Y \subset \R$ choose a $C^\infty$ function $\Phi : \R \rightarrow \R$ vanishing exactly on $Y$. Define the $C^\infty$ submanifold $\widetilde M'$ by
\begin{align*}
\im w' = |z_1'|^2 + |z_1'|^4 + (\re(z_1'))^2 \im(z_2') + \Phi(\re(z_2')) K(z_1',z_2',\re w'),
\end{align*}
where we set $U\coloneqq \widetilde M' \cap \{\re(z_2') \not\in Y\}$ and choose $K$ a $C^\infty$-function in such a way that $U$ does not contain any analytic subvariety.
Then for $t \in Y$ it holds that $H_{0,t}(M) \subset \widetilde M'$, which implies that $Y'\coloneqq \bigcup_{t \in Y} \{H_{0,t}(0)\} \subset E_{\mathcal F}$. We would like to argue that $E_{\mathcal F} = Y'$: Indeed any holomorphic embedding of $M$ into $\widetilde M'$ has the property that its image has empty intersection with $U$, because $U$ does not contain any analytic submanifold.
\end{example}

\bibliography{ref}

\end{document}